\documentclass[10pt]{ijnam}
\usepackage{amsmath,amsfonts,amsthm,amssymb,amsbsy}
\usepackage{subfig}
\newtheorem{thm}{Theorem}[section]
\theoremstyle{definition}

\theoremstyle{plain}
\newtheorem{lem}[thm]{Lemma}
\DeclareMathOperator*{\epi}{\Pi}
\usepackage{mathtools}

\usepackage{graphicx}
\copyrightinfo{2015}{} 

\begin{document}

\title[NC-VEM for convection diffusion reaction equation ]
{Analysis of nonconforming virtual element method for the convection diffusion reaction 
equation with polynomial coefficients}


\author[D. Adak\and  E. Natarajan]{DIBYENDU ADAK \and  E. NATARAJAN}
\address{
  Department of Mathematics,
  Indian Institute of Space Science and Technology ,
  Thiruvananthapuram-695547, INDIA
}
\email{dibyendu.13@iist.ac.in\and thanndavam@iist.ac.in}






\abstract{In this paper we discuss the application of nonconforming virtual element methods(VEM) 
for the second order diffusion dominated convection diffusion reaction equation. Stability
of the virtual element methods has been proved for the symmetric bilinear form. But the same
analysis cannot be carried out for the non-symmetric case. In this work we present
the external virtual element methods using $L^2$ projection operator and prove the 
well-posedness of VEM for non symmetric bilinear form. We also proved polynomial 
consistency of discrete bilinear form assuming $H^2$ regularity of approximate
solution on each triangle. We have shown optimal convergence estimate in the 
broken sobolev norm.}


\maketitle

\section{INTRODUCTION}
\label{intro}
In recent times the virtual element method has been successfully applied to a 
variety of problems\cite{da2013virtual,brezzi2013virtual}. The basic principle of virtual 
element method has been discussed in \cite{de2014nonconforming,beirao2013basic}. A mimetic 
discretization method with arbitrary polynomial order is presented recently 
in \cite{brezzi2009mimetic}, but the classical finite element framework with 
arbitrary polynomial is still making the 
presentation cumbersome \cite{ciarlet2002finite}. The idea of the virtual element
methods is very similar to mimetic finite difference 
methods \cite{droniou2010unified,da2008higher}. Virtual element space is a 
unisolvent space of smooth functions containing a polynomial subspace. In other way
we can say non conforming virtual element method is a generalization of classical 
nonconforming finite element methods. Very recently it has been clearly 
understood that the degrees of freedom associated to trial/test functions is enough to 
construct finite element framework, which lead to the study of virtual element method. Unlike
classical nonconforming FEM \cite{ciarlet2002finite,gao2011note,brenner2008mathematical,fortin1983non}
VEM has the advantage that virtual element space together with polynomial consistency 
property allows us to approximate the bilinear form without explicit knowledge of basis 
function. 

Stability analysis of virtual element method irrespective of conforming or 
non-conforming is quite different from classical finite element method. If the bilinear form 
is symmetric then we divide it into two parts 
\cite{ahmad2013equivalent,de2014nonconforming}, one is responsible for polynomial consistency 
property and other one for stability analysis. The framework for stability analysis for 
conforming and nonconforming virtual element is almost same. A pioneering work using 
elliptic projection operator has been introduced to approximate symmetric bilinear form 
by Brezzi et al in their papers
\cite{da2013virtual,beirao2013basic}. If the bilinear form is not symmetric like 
convection diffusion reaction form, we can not extend this idea directly which may be 
considered as the drawback of using elliptic projection operator 
\cite{ahmad2013equivalent,brezzi2013virtual}. In this case we may use $L^{2}$ projection 
operator for the modified approximation of bilinear form. The name virtual comes from the 
fact that the local approximation space in each mesh either polygon or polyhedra contains 
the space of polynomials together with some non-polynomial smooth function satisfying the 
weak formulation of model problem. The novelty of this method is to take the spaces and the 
degrees of freedom in such a way that the elementary stiffness matrix can be computed without 
actually computing non-polynomial functions, but just using the degrees of freedom.  

In this paper we have approximated non-symmetric bilinear form using $L^{2}$ projection 
operator. We did not approximate the bilinear form same as \cite{de2014nonconforming} to 
avoid difficulty. Conforming virtual element method using $L^{2}$ projection operator has 
been already discussed for convection diffusion reaction equation with variable 
coefficient. We have considered the nonconforming discretization same as defined 
in \cite{knobloch2003p}. In two dimension, the design of schemes of order of 
accuracy $k\geq 1$ was guided by the patch test \cite{irons1972experience,fortin1983non} 
which enforces continuity at $k$ Gauss-Legendre points on edge. Over the last few 
years, further generalization of non-conforming elements have been still considered by 
several authors, one such generalization to stokes problem is considered 
in \cite{stoyan2006crouzeix,arnold1984stable,comodi1989hellan}.

This article is organized as follows. In Section \ref{sec1} we discuss continuous setting 
of the model problem (\ref{temp2}). In Section \ref{nconf} we have discussed construction of non-conforming virtual element 
briefly since it has been discussed in the following paper\cite{de2014nonconforming}. Also, the 
global and local setting of discrete bilinear formulation has been presented explicitly. In 
Section \ref{step2} we have discussed the construction of bilinear form, stability analysis 
and polynomial consistency property of the discrete bilinear form. Section \ref{step3},
discusses the construction of source term, boundary term and Section \ref{step5} 
provides well-posedness and convergence analysis. Stability
analysis discussed in this paper is applicable only for the diffusion dominated problem. Finally in the last section we have 
shown the optimal convergence rate in broken norm $\parallel.\parallel_{1,h}$.

\subsection{Continuous Problem}
\label{sec1}
In this section we present the basic setting and describe the continuous problem. Throughout 
the paper, we use the standard notation of Sobolev spaces \cite{ciarlet2002finite}. Moreover, for 
any integer $l \geq 0$ and a domain $D\in \mathbb{R}^{m} $ with 
$m \leq d,d=2,3 $, $\mathbb{P}^{l}(D) $ is the space of polynomials of degree at most $l$ 
defined on $D$. We also assume the convention that $\mathbb{P}^{-1}=\{0\}$. Let the domain 
$\Omega $ in $\mathbb{R}^{d} $ with $d=2,3$ be a bounded open polygonal domain with straight 
boundary edges for $d=2$ or a polyhedral domain with flat boundary faces for $d=3$. Let 
us consider the model problem:

\begin{flalign}
 -& \nabla \cdot \left(\mathbf{K}(x) \, \nabla(u) \right)+\boldsymbol{\beta}(x) \cdot \nabla u +c(x) \, u = f(x) \quad \text{in} \quad \Omega, \nonumber \\ 
 & u = g  \quad \text{on} \quad \partial \Omega,
\label{temp2} 
\end{flalign}
where $\mathbf{K} \in (C^{1}(\Omega))^{d\times d}$ is the diffusive 
tensor, $\boldsymbol{\beta}(x) \in (C(\Omega))^{d} $ is the convection 
field, $c\in C(\Omega)$ is the reaction field and $f\in L^{2}(\Omega)$. We assume 
that $(c(x)-\frac{1}{2} \nabla \cdot \boldsymbol{\beta}(x))\geq c_{0}$, where $c_{0}$ is a 
positive constant. This assumption guarantees that (\ref{temp2}) admits a unique 
solution. The diffusive tensor is a full symmetric $d\times d$ sized matrix and strongly 
elliptic, i.e. there exists two strictly positive real constants $\xi$ and $ \eta$ such that 
\begin{center}
$\eta \, |\textbf{v}|^{2} \leq \textbf{v} \, \textbf{K} \, \textbf{v} \leq \xi \, |\textbf{v}|^{2} $
\end{center}
for almost every $ x \in \Omega$ and for any sufficiently smooth vector field 
$v$ defined on $\Omega$, where $|\cdot|$ denotes the standard euclidean norm on 
$\mathbb{R}^{d}$. $\mathbf{K},\boldsymbol{\beta},c$ are chosen to be
polynomials for the present problem.

The weak formulation of the model problem (\ref{temp2}) reads:

Find $u \in V_{g} \, \, \text{such} \, \, \text{that}$
\begin{equation*} A(u,v)=<f,v> \quad \forall \  v \in V \end{equation*}
where the bilinear form $A(\cdot,\cdot):V_{g} \times V \rightarrow \mathbb{R}$ is given by 
 \begin{eqnarray}
 A(u,v) & = &\int_{\Omega} \textbf{K} \, \nabla u \cdot \nabla v +\int_{\Omega}(\boldsymbol{\beta} \cdot \nabla u)v+\int_{\Omega}c \, u \, v \nonumber \\
 <f,v> & = &\int_{\Omega} fv
 \label{temp3}
\end{eqnarray}  
$<\cdot,\cdot>$ denotes the duality product between the functional space $V^{\prime}$ and $V$, where the space $V_{g}$ and $V$ are defined by
\begin{eqnarray}
V_{g} & = &\{ v \in H^{1}(\Omega):v|_{\partial \Omega }=g\} \nonumber \\
V & = &H^{1}_{0}(\Omega) 
\label{tempb5}
\end{eqnarray}
We define the elemental contributions of the bilinear form $A(\cdot,\cdot)$ by
\begin{eqnarray}
a(u,v) & = &\int_{\Omega} \textbf{K} \, \nabla u \cdot \nabla v \  d\Omega \\
b(u,v) & = &\int_{\Omega}(\boldsymbol{\beta} \cdot \nabla u)v\  d\Omega \\
c(u,v) & = &\int_{\Omega}c \, u \, v \ d\Omega 
\end{eqnarray}

Now we can estimate
\begin{eqnarray}
A(v,v) & = &\int_{\Omega}\textbf{K} \, \nabla v \cdot \nabla v+\int_{\Omega}(\boldsymbol{\beta} \cdot \nabla v) \, v+\int_{\Omega}c \, v^{2} \nonumber\\
 & \geq &\eta \, \| \nabla v \|^2_0+\int(c-\frac{1}{2}\nabla \cdot \boldsymbol{\beta}) \, v^{2} \nonumber \\
 & \geq & C \, \| v \|^2_1
 \label{temp4}
\end{eqnarray}
Combining inequality (\ref{temp4}) with the continuity of $A(\cdot, \cdot)$ it follows that there 
exists a unique solution to the variational form of equation (\ref{temp2}).    

\subsection{Basic Setting}
We describe now the basic assumptions of the mesh partitioning and related  
function spaces. Let $\{ \tau_{h}\}$ be a family of decompositions of $\Omega $ into 
elements, $T$ and $ \varepsilon_{h}$ denote a single element and set of edges of a 
particular partition respectively. By $\varepsilon^{0}_{h}$ and 
$\varepsilon^{\partial}_{h}$ we refer to the set of interior and boundary 
edges/faces respectively. We will follow the same assumptions on the family of partitions 
as \cite{de2014nonconforming}.   

\subsection{Assumptions on the family of partitions $\{\tau_{h} \}$}
There exists a positive $\rho >0 $ \ such that
\begin{itemize}
\item[(A1)] for every element $ T$ and for every edge/face $e\subset \partial T$, we 
have $h_{e}\geq \rho \, h_{T} $.
\item[(A2)] every element $T$ is star-shaped with respect to all the points of a sphere of 
radius $\geq \rho \, h_{T}$.
\item[(A3)] for $d=3$, every face $e \in \varepsilon_{h}$ is star-shaped with respect to all the 
points of a disk having radius $\geq \rho \, h_{e} $.
\end{itemize}
The maximum of the diameters of the elements $T \in \tau_{h} $ will be denoted by $h$. For 
every $h>0$, the partition $\tau_{h}$ is made of a finite number of polygons or polyhedra.

We introduce the broken sobolev space for any $s>0$
\begin{equation}
 H^{s}(\tau_{h})=\epi_{T \in \tau_{h}} H^{s}(T)=\left\{v \in L^{2}(\Omega):v|_{T} \in H^{s}(T) \right\}
 \label{temp5}
\end{equation}
and define the broken $ H^{s}$-norm 
\begin{equation}
\| v \|^{2}_{s,\tau_{h}}=\sum_{T \in \tau_{h}} \| v \|^{2}_{s,T} \quad \forall v \in H^{s}(\tau_{h})
\label{temp6}
\end{equation}
In particular for $s=1$
\begin{equation}
\| v \|^{2}_{1,\tau_{h}}=\sum_{T \in \tau_{h}} \| v \|^{2}_{1,T} \quad \forall v \in H^{1}(\tau_{h}) \nonumber
\end{equation}

Let $e \in \varepsilon^{0}_{h}$ be an interior edges and let $T^{+}, T^{-} $ be two 
triangles which share $e$ as a common edge. We denote the unit normal on $e$ in the 
outward direction with respect to $T^{\pm}$ by $ n_{e}^{\pm}$. We then define the 
jump operator as:
\begin{center}
$[\vert v \vert]$ := \, $v^{+} n_{e}^{+} + v^{-}n_{e}^{-}$ \quad on  \quad $ e \in \varepsilon^{0}_{h}$
\end{center}
\begin{center}
and $[\vert v \vert]$ := \, $v \, n_{e}$ \quad on \quad $e \in \varepsilon^{\partial}_{h}$
\end{center}
\subsection{Discrete space}
In this section we will introduce discrete space same as \cite{de2014nonconforming}. For 
an integer $k \geq 1$ we define 
\begin{equation}
H^{1,nc}(\tau_{h};k)=\left\{ v \in H^{1}(\tau_{h}): \int_{e}[| v |] \cdot n_{e} \, q \, ds=0 \ \forall \ q \in P^{k-1}(e), \ \forall \ e \in \varepsilon_{h} \right\} 
\label{temp7}
\end{equation}
We mention that if $v \in H^{1,nc}(\tau_{h};k) $ then $v $ in said to satisfy patch 
test \cite{knobloch2003p} of order $k$. To approximate second order problems we must satisfy the patch test. The 
space $H^{1,nc}(\tau_{h};1)$ is the space with minimal required order of patch test 
 to ensure convergence analysis.
\subsection{Nonconforming virtual element methods}
\label{nconf}
In this section we explain discrete nonconforming VEM framework for the equation
(\ref{temp2}). Before passing from the weak formulation to discrete problem, we first 
apply integration by parts to the convective term $(\boldsymbol{\beta} \cdot \nabla u,v)$ to 
obtain

\begin{equation}
\int_{\Omega} (\boldsymbol{\beta} \cdot \nabla u) v=\frac{1}{2} \left[\int_{\Omega} (\boldsymbol{\beta} \cdot \nabla u) v -\int_{\Omega} (\boldsymbol{\beta} \cdot \nabla v) u -\int_{\Omega}\nabla \cdot \boldsymbol{\beta} \, u \, v \right] \nonumber
\end{equation}
for $u \in H^{1}(\Omega),v\in H^{1}_{0}(\Omega)$ 
 
Bilinear form (\ref{temp3}) can be written as 
\begin{equation}
A(u,v)=\sum_{T \in \tau_{h}} A^{T}(u,v) 
\label{temp8}
\end{equation}
where
\begin{flalign}
A^{T}(u,v) & =\int_{T} \textbf{K} \, \nabla u \cdot \nabla v +b^{\text{conv}}_{T}(u,v)+\int_{T} c \, u \, v \nonumber \\
\text{and} \quad b^{\text{conv}}_{T}(u,v) & = \frac{1}{2} \left[\int_T (\boldsymbol{\beta} 
\cdot \nabla u) v -\int_T (\boldsymbol{\beta} \cdot \nabla v) u - 
\int_{T} \nabla \cdot \boldsymbol{\beta} \, u \, v \right]
\label{temp9}
\end{flalign}
We want to construct a finite dimensional space $ V^{k}_{h} \subset H^{1,nc}(\tau_{h};k) $, a 
bilinear form $A_{h}(\cdot , \cdot):V_{h, g}^{k} \times V_{h}^{k} \rightarrow \mathbb{R}$, and 
an element $f_{h} \in (V^{k}_{h})^{\prime}$ such that the discrete problem:
\linebreak

Find $u_{h} \in V_{h,g}^{k}$ such that 
\begin{equation}
A_{h}(u_{h},v_{h})=<f_{h},v_{h}> \  \ \forall \ v_{h}\in V_{h}^{k}
 \label{temp10} 
 \end{equation} 
\quad has a unique solution $u_{h}$.
\subsection{Local nonconforming virtual element space:}
We define for $k\geq 1$ the finite dimensional space $V^{k}_{h}(T)$ on $T$ as
\begin{equation}
V^{k}_{h}(T)=\left\{v \in H^{1}(T):\frac{\partial v}{\partial \mathbf{n}} \in \mathbb{P}^{k-1}(e),v \mid_{e}\in \mathbb{P}^{k}(e)   \ \forall \ e \subset \partial T, \Delta v \in \mathbb{P}^{k-2}(T) \right\}
\label{temp11}
\end{equation}
with the usual convention that $\mathbb{P}^{-1}(T)=\{0\} .$ \\
  
From the definition, it immediately follows that $\mathbb{P}^{k}(T) \subset 
V^{k}_{h}(T)$. The non-conforming VEM is formulated through the $L^2$ projection 
operators,
\begin{equation}
\Pi_{k}:V^{k}_{h}(T)\rightarrow \mathbb{P}^{k}(T) \nonumber
\end{equation}  

\begin{equation}
\Pi_{k-1}:\nabla\left(V^{k}_{h}(T)\right) \rightarrow \left(\mathbb{P}_{k-1}(T)\right)^{d} \nonumber.
\end{equation}

For $k=1$, degrees of freedom are defined same as Crouzeix-Raviart 
element \cite{crouzeix1973conforming,brenner2008mathematical}. In general we can say 
normal derivative $\frac{\partial u}{\partial n}$ of an arbitrary element of
$V^{1}_{h}$ is constant on each edge $e\subset \partial T$ 
(and different on each edges)and inside $T$ are harmonic(i.e., $\Delta v=0$). It can 
be easily concluded that the total no of degrees of freedom of a particular element $T$ 
with $n$ edges/faces is $n$ which is explained in detail in \cite{de2014nonconforming}.

For $k=2$, the space $V^{2}_{h}(T)$  consists of functions whose normal 
derivative $\frac{\partial u}{\partial n}$ is a polynomial of degree $1$ on each 
edge/face, i.e. $\frac{\partial u}{\partial n} \subset \mathbb{P}^{1}(e)$ and is a 
polynomial of degree $0$ on interior region, i.e. a constant function. The dimension 
of $V^{2}_{h}(T)$ is $d \, n+1$ where $n, d$ denote number of edges/face associated with 
an element $T$ and spatial dimension of $T$ respectively.

For each element $T$,the dimension of $V^{k}_{h}(T)$ is given by 
\begin{equation*}
N_{T} =
\begin{cases}
\displaystyle nk+\frac{(k-1)k}{2} & \text{for} \quad d=2, \\
\displaystyle \frac{nk(k+1)}{2}+\frac{(k-1)k(k+1)}{6} & \text{for} \quad d=3
\end{cases}
\end{equation*}   
which is explained in detail in \cite{de2014nonconforming,brezzi2013virtual}. We 
need to introduce some further notation to define degrees of freedom same 
as \cite{de2014nonconforming}. Let us define space of scaled monomials 
$M^{l}(e)$ and $M^{l}(T)$ on $e$ and $T$ as 
\begin{equation}
M^{l}(e)=\left\{\left(\frac{(\mathbf{x}-\mathbf{x}_{e})}{h_{e}}\right)^{s}, \vert s \vert \leq l \right\} \nonumber
\end{equation}    
\begin{equation}
M^{l}(T)=\left\{\left(\frac{(\mathbf{x}-\mathbf{x}_{T})}{h_{T}}\right)^{s}, \vert s \vert \leq l \right\} \nonumber
\end{equation}
where $s=(s_{1},s_{2}, \cdots, s_{d})$ be a $d$-dimensional multi index notation with 
$\vert s \vert =\sum_{i=1}^{d}s_{i}$ and $\displaystyle \mathbf{x}^{s}=\Pi_{i=1}^{d} x_{i}^{s_{i}}$ 
where $\mathbf{x}=(x_{1},\dot{...}, x_{d}) \in \mathbb{R}^{d} $ and $l\geq 0$ be an integer.
In $V^{k}_{h}(T)$ we will choose same degrees of freedom as defined in \cite{de2014nonconforming}.
On each edge $e\subset \partial T$
\begin{equation}
\mu^{k-1}_{e}(v_{h})=\left\{ \frac{1}{e} \int_{e} v_{h} \, q \, ds,\quad \forall q \in M^{k-1}(e)\right\}
\label{temp12} 
\end{equation}
on each element $T$
\begin{equation}
\mu^{k-2}_{T}(v_{h})= \left\{ \frac{1}{\vert T \vert} \int_{T} v_{h} \, q, \quad \forall q \in M^{k-2}(T) \right\}
\label{temp13}
\end{equation}
The set of functional defined in (\ref{temp12}) and (\ref{temp13}) are unisolvent for 
the space $V^{k}_{h}(T)$
\begin{lem}
Let T be a simple polygon/polyedra with $n$ edges/faces, and let $V^{k}_{h}(T)$ be the space 
defined in (\ref{temp11}) for any integer $k \geq 1$.The degrees of freedom (\ref{temp12}) 
and(\ref{temp13}) are unisolvent for $V^{k}_{h}(T)$.
\end{lem}

\begin{proof}
See the details in \cite{de2014nonconforming}.
\end{proof}
The degrees of freedom equation (\ref{temp12}) and (\ref{temp13}) are defined by using 
the monomials in $\mu^{k-1}_{e}$ and $\mu^{k-2}_{T}$ as basis functions for the polynomial 
spaces $\mathbb{P}^{k-1}(e)$ and $\mathbb{P}^{k-2}(T)$. This special chioce of the basis 
functions gives advantages to implement the nonconforming VEM on arbitrary polygonal 
domain. Implementation part is described explicitly in articles 
\cite{ahmad2013equivalent,da2013virtual}.

\subsection{Global nonconforming virtual element space } 
We now introduce the nonconforming(global) virtual element space $V^{k}_{h}$ of order $k$. We 
have already defined local nonconforming virtual element space $V^{k}_{h}$(T) on each 
element $T$ of partition $\tau_{h}$. The global nonconformi
ng virtual element space 
$V^{k}_{h}$ of order $k$ is defined by
\begin{equation}
V^{k}_{h}=\left\{ v_{h}\in H^{1,nc}(\tau_{h};k):v_{h}|_{T} \in V^{k}_{h}(T)\quad	\forall \, T \in \tau_{h} \right\}
\label{temp14}
\end{equation}

\subsection{Interpolation error}
We can define an interpolation operator in $V^{k}_{h}$ having optimal approximation 
properties using same idea as described in 
\cite{brezzi2013virtual,ciarlet1991basic,ciarlet2002finite,de2014nonconforming}. We can 
define an operator $\chi_{i}$ which associates each function $\phi$ to the $i^{th}$ degree 
of freedom and virtual basis functions $\psi_{i} $ of global virtual element space satisfies 
the condition $\chi_{i}(\psi_{i})=\delta_{ij}$ for $i,j=1,2,\cdots,N $ where $N$ denotes 
the number of degrees of freedom of global space. Then for 
any $ v \in H^{1,nc}(\tau_{h},k)$, there exists unique $ v_{I} \in V^{k}_{h}$ such that 
\begin{equation}
\chi_{i}(v-v_{I})=0 \quad \forall \ i=1,2,\cdots,N \nonumber
\end{equation}  
     
     Using all these properties we can claim that there exists a constant $C>0$, independent 
of $h$ such that for every $h>0$, every $K \in \tau_{h}$, every $s$ with $2\leq s\leq k+1 $ 
and every $v \in H^{s}(K)$ the interpolant 
     $v_{I} \in V^{k}_{h}$ satisfies:
\begin{equation}
\| v-v_{I} \|_{0,T}+h_{T} \| v-v_{I} \|_{1,K} \leq C h^{s}_{T} \| v \|_{s,T}
\label{temp15}
\end{equation}
Technical detail of the above approximation is described in \cite{de2014nonconforming}.

\section{Construction of $A_{h}$}
\label{step2}
The goal of this section is to define the nonconforming  virtual element discretization 
(\ref{temp10}). If the discretized bilinear form is symmetric then we can easily prove the 
good stability and nice approximation property and ensure the computability of the defined 
bilinear form $A_{h}(\cdot,\cdot)$ over functions in $V^{k}_{h}$. But since the bilinear 
form of the convection diffusion reaction equation is not symmetric we will accept certain 
assumptions on the model problem (\ref{temp2}). Diffusive and reactive part are symmetric 
and therefore we will split these two terms as a sum of polynomial part or consistency 
part and stability part. Convective part is not symmetric therefore we will take only 
polynomial approximation of this part. The present framework is only applicable to 
the diffusion-dominated case when the Peclet number is sufficiently small.

\begin{equation}
A_{h}(u_{h},v_{h})=\sum_{T \in \tau_{h}} A^{T}_{h}(u_{h},v_{h}) \quad \forall\ u_{h},v_{h} \in V^{k}_{h},
\label{temp16}
\end{equation}
where $A^{T}_{h}:V^{k}_{h}\times V^{k}_{h} \rightarrow \mathbb{R} $ denoting the restriction 
to the local space $V^{k}_{h}(T)$. The bilinear form $ A^{T}_{h} $ can be decomposed into 
sum of element terms. Thus defining the approximate bilinear form 
\begin{equation}
A^{T}_{h}(u_{h},v_{h}):=a^{T}_{h}(u_{h},v_{h})+b^{T}_{h}(u_{h},v_{h})+c^{T}_{h}(u_{h},v_{h}) 
\label{temp16}
\end{equation}     
for each element $T$, we define the element contributions to $A^{T}_{h}$ by 
\begin{eqnarray}
 a^{T}_{h}(u_{h},v_{h}) & := &  \int_{T} \mathbf{K} \, \Pi_{k-1}(\nabla u_{h}) \cdot \Pi_{k-1}(\nabla v_{h}) \, dT \nonumber \\
                        & + &\mathbf{S}^{T}_{a}\left((I-\Pi_{k})u_{h},(I-\Pi_{k})v_{h} \right)  \nonumber\\
 b^{T}_{h}(u_{h},v_{h}) & := & \frac{1}{2} \Big( \int_{T}\boldsymbol{\beta} \, \Pi_{k-1}(\nabla u_{h}) \, \Pi_{k}(v_{h}) \, dT \nonumber  \\
                        &- &\int \boldsymbol{\beta} \cdot \Pi_{k-1}(\nabla v_{h}) \, \Pi_{k}(u_{h}) \, dT \nonumber  \\            
                        & - &\int_{T}(\nabla \cdot \boldsymbol{\beta}) \, \Pi_{k}(u_{h}) \, \Pi_{k}(v_{h}) \, dT \Big) \nonumber \\  
 c^{T}_{h}(u_{h},v_{h}) & := & \int_{T} {c} \, \Pi_{k}(u_{h}) \, \Pi_{k}(v_{h}) dT \nonumber \\
                        &+ & \mathbf{S}^{T}_{c}\left((I-\Pi_{k})u_{h},(I-\Pi_{k})v_{h}\right) 
\label{temp17}
\end{eqnarray}  
where $\mathbf{S}^{T}_{a}$ and $\mathbf{S}^{T}_{c}$ are the stabilising terms. These terms 
are symmetric and positive definite on the quotient space 
$V^{k}_{h}(T)\diagup \mathbb{P}_{k}(T) $ and satisfy the stability property:

\begin{center}
$\alpha_{\ast} a^{T}(v_{h},v_{h})\leq \mathbf{S}^{T}_{a}(v_{h},v_{h}) \leq \alpha^{\ast}a^{T}(v_{h},v_{h}),$
\label{temp18}
\end{center}

\begin{center}
$\gamma_{\ast}c^{T}(v_{h},v_{h}) \leq \mathbf{S}^{T}_{c}(v_{h},v_{h}) \leq \gamma^{\ast}c^{T}(v_{h},v_{h}), $
\label{temp19}
\end{center}
for all $v_{h}\in V^{k}_{h} $ with $\Pi_{k}(v_{h})=0 $. The first term ensures polynomial 
consistency property and second term ensures stability property of the corresponding bilinear 
form $ a^{T}_{h}(u_h, v_h)$ and $c^{T}_{h}(u_h,v_h)$.
 
\subsection{Consistency}
\begin{lem}
Let ${u_{h}} |_{T} \in \mathbb{P}^k(T)$ and $v_{h} |_{T} \in H^2(T)$, then the 
bilinear forms $ a^{T}_{h}, b^{T}_{h}, c^{T}_{h}$ defined in equation (\ref{temp17}) 
satisfy the following consistency property for all $h>0$ and for all $T \in \tau_{h}$ .  
\end{lem}

\begin{proof}
Whenever either $u_{h}$ or $ v_{h}$ or both are elements of the polynomial space 
$\mathbb{P}_{k}(T)$, the following consistency property satisfy

\begin{eqnarray}
a^{T}_{h}(u_{h},v_{h})& = & a^{T}(u_{h},v_{h}) \nonumber\\
b^{T}_{h}(u_{h},v_{h}) & = & b^{T}(u_{h},v_{h}) \nonumber \\
c^{T}_{h}(u_{h},v_{h}) &= & c^{T}(u_{h},v_{h})  
\label{temp20}
\end{eqnarray}

The consistency property in (\ref{temp20}) follows immediately since 
$\Pi_{k}( \mathbb{P}^{k}(T))=\mathbb{P}^{k}(T) $ which implies 
$\mathbf{S}^{T}_{a}\left(p-\Pi_{k}(p),v_{h}-\Pi_{k}(v_{h})\right)=0$ and 
$\mathbf{S}^{T}_{c}\left(p-\Pi_{k}(p),v_{h}-\Pi_{k}(v_{h})\right)=0$. Now we will prove  
$a^{T}_{h}(p,v_{h}) = a^{T}(p,v_{h}),  \, b^{T}_{h}(p,v_{h}) = b^{T}(p,v_{h}), \, c^{T}_{h}(p,v_{h}) =  c^{T}(p,v_{h})$ 
for all $p \in \mathbb{P}^{k}(T)$ and for all $v_{h}\in V^{k}_{h}(T)$.

\begin{eqnarray}
a^{T}_{h}(p,v_{h}) &= & \int_{T} \mathbf{K} \, \nabla p \cdot \Pi_{k-1}(\nabla v_{h}) \,  dT \nonumber \\ \nonumber
                   & = & \int_{T} \left(\Pi_{k-1}(\nabla v_{h})-\nabla v_{h} \right) \, \mathbf{K} \, \nabla p \,  dT \\
                   & + & \int_{T} \nabla v_{h} \, \mathbf{K} \, \nabla p \, \nonumber \,  dT \\
                   & = & \int_{T} \mathbf{K} \, \nabla p \, \nabla v_{h} \nonumber \, dT \\ 
                   & = & a^{T}(p,v_{h})
                   \label{temp21}      
\end{eqnarray}

\begin{eqnarray}
b^{T}_{h}(p,v_{h}) & = & \frac{1}{2} \Big(\int_{T} \boldsymbol{\beta} \cdot \nabla p \, \Pi_{k}(v_{h}) dT \nonumber \\
                   & - &  \int_{T} \boldsymbol{\beta} \cdot \Pi_{k-1}(\nabla v_{h}) \, p \, dT-\int_{T}(\nabla. \boldsymbol{\beta}) \, p \, \Pi_{k}(v_{h}) \, dT \Big) 
                   \label{tempc1}
\end{eqnarray}

\begin{eqnarray}
\int_{T} \boldsymbol{\beta} \cdot \nabla p \, \Pi_{k}(v_{h}) \, dT & = & \int_{T} (\Pi_{k}(v_{h})-v_{h}) \, \nabla p \cdot \boldsymbol{\beta} \, dT \nonumber\\
                                                        & + & \int_{T} \boldsymbol{\beta} \cdot \, \nabla p \, v_{h} \nonumber \\
                                                        & = & \int_{T} \boldsymbol{\beta} \cdot \nabla p \, v_{h}  
                                                        \label{tempc2}
\end{eqnarray}

\begin{eqnarray}
\int_{T} \boldsymbol{\beta} \cdot \Pi_{k-1}(\nabla v_{h}) \, p &=& \int_{T} (\Pi_{k-1}(\nabla v_{h})-\nabla v_{h}) \cdot \boldsymbol{\beta} \, p \nonumber \\
                                                      & + &\int_{T} \boldsymbol{\beta} \cdot \nabla v_{h} \, p \nonumber \\
                                                      & \leq & \parallel \boldsymbol{\beta} \parallel_{\infty,T} \parallel p \parallel_{0,T} \parallel \nabla v_{h}-\Pi_{k-1}(\nabla v_{h}) \parallel +\int_{T} \boldsymbol{\beta} \cdot \nabla v_{h} \, p \nonumber \\
                                                      & \leq & C \parallel \boldsymbol{\beta} \parallel_{\infty,T} h_{T} \vert \nabla v_{h} \vert_{1,T} \parallel p \parallel_{0,T} \nonumber \\
                                                      & +& \int_{T} \boldsymbol{\beta} \cdot \nabla v_{h} \, p \nonumber \\
                                                      & \approx & \int_{T} \boldsymbol{\beta} \cdot \nabla v_{h} \, p  \quad (\, \text{for small values of} \, \, h_{T})
                                                      \label{tempc3}
\end{eqnarray}

\begin{eqnarray}
\int_{T} (\nabla \cdot \boldsymbol{\beta}) \, p \, \Pi_{k}(v_{h}) \, dT &=& \int_{T} (\nabla \cdot \boldsymbol{\beta})\left(\Pi_{k}(v_{h})-v_{h}\right) \, p +\int_{T} (\nabla \cdot \boldsymbol{\beta}) p \, v_{h} \nonumber\\
                                                          &=& \int_{T} (\nabla \cdot \beta) \, p \, v_{h} 
                                                          \label{tempc4}
\end{eqnarray}
Putting estimations (\ref{tempc2}), (\ref{tempc3}) and (\ref{tempc4}) in (\ref{tempc1})
we get $b^{T}_{h}(p,v_{h})=b^{T}(p,v_{h})$

Similarly,
\begin{equation}
c^{T}_{h}(p,v_{h})=c^{T}(p,v_{h})
\label{tempc5}
\end{equation}

Hence we proved required polynomial consistency of local discrete bilinear 
form  $A^{T}_{h}(u_{h},v_{h})$,  i.e.
\begin{equation}
A^{T}_{h}(p,v_{h})=A^{T}(p,v_{h})
\label{tempc6}
\end{equation}
for $p \in \mathbb{P}^{k}$ and $ v_{h} \in V^{k}_{h}$
\end{proof} 



\subsection{Discrete stability} 
Before discussing stability property of the discrete bilinear form 
$A^{T}_{h}(u_{h},v_{h})$ we reveal that the following framework is applicable for 
diffusion dominated case. The stabilizing part $\mathbf{S}^{T}_{a}$ and $\mathbf{S}^{T}_{c}$ 
of the discrete bilinear form(\ref{temp17}) ensure the stability of the bilinear 
form, precisely we can conclude that there exist two pairs of positive 
constants $\alpha_{\ast} $,  $\alpha^{\ast}$ and $\gamma^{\ast}, \gamma_{\ast} $ that are 
independent of $h$ such that 

\begin{equation}
 \alpha_{\ast}a^{T}(v_{h},v_{h})\leq a^{T}_{h}(v_{h},v_{h}) \leq \alpha^{\ast}a^{T}(v_{h},v_{h})
\label{temp22}
\end{equation}

\begin{equation}
 \gamma_{\ast}c^{T}(v_{h},v_{h})\leq c^{T}_{h}(v_{h},v_{h}) \leq \gamma^{\ast}c^{T}(v_{h},v_{h})
\label{temp23}
\end{equation}
 for all $v_{h} \in V^{k}_{h}(T)$ and mesh elements $T$.
\section{Construction of right hand side term}
\label{step3}
In order to build the the right hand side $ <f_{h},v_{h}>$ for $v_{h}\in V^{k}_{h}$ we need 
polynomial approximation of degree $(k-2)\geq 0$,that is  $f_{h}=\textbf{P}^{T}_{k-2}f$ on 
each $T \in \tau_{h}$, where $\textbf{P}^{T}_{k-2}$ is $L^{2}(T)$ projection operator 
on $\mathbb{P}^{k-2}(T)$ for each element $T \in \tau_{h}$. We define $f_{h}$ locally by:
 
 \begin{equation*}
 (f_{h})|_{T}:=
 \begin{cases}
 \textbf{P}^{T}_{0}(f) & \text{for $k=1 $} \\
 \textbf{P}^{T}_{k-2}(f)& \text{ for $k\geq 2 $}
 \label{temp24}
 \end{cases}
 \end{equation*}
 The projection operator is orthogonal to the polynomial space $\mathbb{P}^{k}(T)$. Therefore 
 we can write as 
 
 \begin{equation}
 <f_{h},v_{h}> :=\sum_{T} \int_{T} \mathbf{P}^{T}_{k-2}(f) \, v_{h} \, dT=\sum_{T}\int_{T}f \, \, \mathbf{P}^{T}_{k-2}(v_{h}) \, dT
 \end{equation}
Now we can prove the error estimates using orthogonality property of 
projection operator, Cauchy-Schwarz inequality and standard 
approximates \cite{ciarlet2002finite,ciarlet1991basic} for $k\geq2$, $s \geq 1.$
 
\begin{eqnarray}
 \left| <f,v_{h}>-<f_{h},v_{h}> \right| & = & \left| \sum_{T} \int_{T}\left(f-\mathbf{P}^{T}_{k-2}(f)\right) \, v_{h} \, dT \right| \nonumber\\
                                     & = &\left| \sum_{T} \int_{T}(f-\mathbf{P}^{T}_{k-2}(f))(v_{h}-\mathbf{P}_{0}^{T}(v_{h})) \, dT \right| \nonumber \\
                                    & \leq & \parallel (f-\mathbf{P}^{T}_{k-2}(f)) \parallel_{0,\tau_{h}} \parallel (v_{h}-\mathbf{P}^{T}_{0}(v_{h}))\parallel_{0,\tau_{h}} \nonumber \\
                                    & \leq & C h^{\text{min}(k,s)} \vert f \vert_{s-1,h} \vert v_{h} \vert_{1,h}
\label{temp25}
 \end{eqnarray} 
 For $k=1$ the above analysis is not applicable, so we do the following
\begin{equation}
\tilde{v}_{h}|_{K}:=\frac{1}{n} \sum_{e \in \partial K} \frac{1}{\vert e \vert} \int_{e} v_{h} \, ds \approx \mathbf{P}^{T}_{0}(v_{h}),
\label{temp26}
\end{equation}
\begin{equation}
<f_{h},\tilde{v}_{h}>:=\sum_{T}\int_{T} \mathbf{P}^{T}_{0}(f) \, \tilde{v}_{h} \approx \sum_{T} \vert T \vert \, \mathbf{P}^{T}_{0}(f) \, \mathbf{P}^{T}_{0}(v_{h})
\end{equation} 

\begin{eqnarray}
 \left|<f,v_{h}>-<f_{h},\tilde{v}_{h}> \right|  &= & \left| \sum_{T} \int_{T}(f v_{h}-\mathbf{P}^{T}_{0}(f)\tilde{v}_{h}) \right| \nonumber \\
                                            &\leq & \left| \sum_{T} \int_{T}(f-\mathbf{P}_{0}^{T}(f))v_{h} \right|+ \left| \sum_{T} \int_{K} \mathbf{P}_{0}^{T}(f)(v_{h}-\tilde{v}_{h}) \right| \nonumber  \\
                                            &= & \left| \sum_{T} \int_{T}(f-\mathbf{P}_{0}^{T}(f))v_{h} \right| \nonumber \\
                                            & \leq & C h \vert f \vert_{0,h} \vert v_{h} \vert_{1,h} 
 \label{temp27}
\end{eqnarray}

\subsection{Construction of the boundary term. }
We define $g_{h}:=\mathbf{P}^{e}_{k-1}(g)$ where $g$ is non-homogeneous Dirichlet
boundary value.
\begin{equation}
\int_{\varepsilon^{\partial}_{h}} g_{h} \, v_{h} \, ds:=\sum_{e \in {\varepsilon^{\partial}_{h}}} \int_{e} \mathbf{P}^{e}_{k-1}(g) \, v_{h} \, ds=\sum_{e \in {\varepsilon^{\partial}_{h}}} \int_{e} g \, \mathbf{P}^{e}_{k-1}(v_{h})  \, ds \quad \forall \ v_{h}\in V^{k}_{h}
\end{equation} 
The above estimation guides us to compute the boundary terms easily using 
degrees of freedom.

\subsection{Estimation of the jump term}
\begin{equation}
  J_{h}(u,v_{h})=\sum_{T \in \tau_{h}} \int_{\partial T} \frac{\partial u}{\partial \mathbf{n}_{T}} \, v \, ds=\sum_{e \in \varepsilon_{h}} \int_{e} \nabla u \cdot [\vert v_{h}\vert ]
  \label{temp28}
  \end{equation}
  \begin{lem}
  Let (A1,A2,A3) be satisfied, $k\geq 1$ and $ u \in H^{s+1}(\Omega)$ with $ s\geq 1$ be the 
  solution of the model problem (\ref{temp2}). Let 
  $v_{h}\in H^{1,nc}(\tau_{h};1)$ be an arbitrary function. Then, there exists a 
  constant $C>0 $ independent of $h$ such that 
  \begin{equation}
  \vert J_{h}(u,v_{h})\vert  \leq C h^{\text{min}(s,k)} \parallel u \parallel_{s+1,\Omega} \vert v_{h} \vert_{1,h}
  \label{tempb6}
\end{equation}   
  \end{lem}
  
  \begin{proof}
  Let $v_{h} \in H^{1,nc}(\tau_{h};k)$ be an arbitrary element. From the definition of the 
  finite element space $H^{1,nc}(\tau_{h};k)$, we can say $v_{h}$ satisfies patch test of 
  order $k$. Hence the following equality holds
  \begin{equation}
  \int_{e} [\vert v_{h}\vert ] \, q \, ds=0, \quad \forall q \in \mathbb{P}_{k-1}
  \label{temp29}
\end{equation} 

Let $\mathbf{P}^{e}_{k}:L^{2}(e)\rightarrow \mathbb{P}^{k}(e) $ is the $L^{2}$-orthogonal 
projection operator onto the space $\mathbb{P}^{k}(e)$ for $k\geq 1$. Using patch test of 
order $k$, Cauchy-Schwarz inequality and $L^{2}(e)$ orthogonal projection 
operator $\mathbf{P}^{e}_{k}$ we can find
\begin{eqnarray}
J_{h}(u,v_{h}) & = &\sum_{T \in \tau_{h}} \int_{\partial T} \frac{\partial u}{\partial \mathbf{n}_{T}} v ds\nonumber\\
               &= &\sum_{e \in \varepsilon_{h}} \int_{e} \nabla u.[\vert v_{h}\vert ]\nonumber \\
               & = &\vert \sum_{e \in \varepsilon_{h}} \int_{e} (\nabla u-\mathbf{P}^{e}_{k-1}(\nabla u)).[\vert v_{h}\vert] ds \vert \nonumber \\
               & = &\vert\sum_{e \in \varepsilon} \int_{e} (\nabla u-\mathbf{P}^{e}_{k-1}(\nabla u)).([\vert v_{h}\vert]-\mathbf{P}^{e}_{0}([\vert v_{h}\vert]))\vert \nonumber \\
               & \leq & \sum_{e \in \varepsilon_{h}} \parallel \nabla u-\mathbf{P}^{e}_{k-1}(\nabla u)\parallel_{0,e} \parallel [\vert v_{h}\vert]-\mathbf{P}^{e}_{0}([\vert v_{h}\vert])\parallel_{0,e}
               \label{temp29}
\end{eqnarray}
using standard polynomial approximation on edge $e$
\begin{align}
\parallel \nabla u-\mathbf{P}^{e}_{k-1}(\nabla u)\parallel_{0,e} & \leq C h^{min(m,k)-\frac{1}{2}} \parallel u \parallel_{m+1,T} \label{tempb7}\\
\parallel [\vert v_{h} \vert]-\mathbf{P}^{e}_{0}([\vert v_{h} \vert])\parallel_{0,e} & \leq C h^{\frac{1}{2}} \vert v_{h} \vert_{1,T}
  \label{temp30}
\end{align}
we can easily bound the above two terms. Hence, putting the 
estimation equation (\ref{tempb7}), (\ref{temp30}) in (\ref{temp29}) we obtain required result.   
\end{proof}
  
\section{Well-posedness of nonconforming-virtual element methods}
\label{step4}
In this section we will discuss the well-posedness of nonconforming-virtual element 
method. Let the assumption (A1,A2,A3), polynomial consistency, stability defined 
in equation (\ref{temp22}), (\ref{temp23}) holds then the bilinear form $A_{h}$ is coercive 
with respect to broken-norm $\parallel .\parallel_{1,h} $,  i.e.
  
\begin{equation}
  A_{h}(v_{h},v_{h}) \geq \alpha \parallel v_{h} \parallel_{1,h}^{2} \quad \forall v_{h} \in V_{h}
  \label{temp31}
\end{equation}    
where $\alpha$ is a positive constant. 


Using the stability properties equation (\ref{temp22}),(\ref{temp23}) for diffusion and 
reaction parts of discrete bilinear form we can bound 

\begin{eqnarray}
A^{T}_{h}(v_{h},v_{h}) & \geq & \alpha_{\ast} \, a^{T}(v_{h},v_{h})+b_{h}^{T}(v_{h},v_{h})+\gamma_{\ast} \, c^{T}(v_{h},v_{h}) \nonumber \\
                       & \geq & \alpha_{\ast} \, \eta \, \vert v_{h}\vert^{2}_{1,T} +b^{T}(v_{h},v_{h})+\gamma_{\ast} \, c^{T}(v_{h},v_{h})+[b^{T}_{h}(v_{h},v_{h})-b^{T}(v_{h},v_{h})] \nonumber \\
                       &\geq & \alpha_{\ast} \, \eta \, \vert v_{h} \vert^{2}_{1,T}+\text{min}(1,\gamma_{\ast})(b^{T}(v_{h},v_{h}) \nonumber \\
                       &+ &c^{T}(v_{h},v_{h}))+[b^{T}_{h}(v_{h},v_{h})-b^{T}(v_{h},v_{h})] \nonumber \\ 
                       & \geq & \alpha_{\ast} \, \eta \, \vert v_{h} \vert^{2}_{1,T}+\text{min}(1,\gamma_{\ast}) \, c_{0} \parallel v_{h} \parallel^{2}_{0,T}  \nonumber \\ 
                       & - & \vert b^{T}_{h}(v_{h},v_{h})-b^{T}(v_{h},v_{h})\vert \quad \forall v_{h} \in V^{k}_{h}(T)
\label{temp32}
\end{eqnarray}  

\begin{eqnarray}
\vert b^{T}_{h}(v_{h},v_{h})-b^{T}(v_{h},v_{h}) \vert & =& \frac{1}{2} \left| \int_{T}(\nabla \cdot \boldsymbol{\beta})(\Pi_{k}(v_{h}))^{2} dT -\int_{T} (\nabla.\boldsymbol{\beta}) \, v_{h}^{2} \, dT \right| \nonumber \\
                                                      &=& \frac{1}{2} \Big| \int_{T}(\nabla \cdot \boldsymbol{\beta})(\Pi_{k}(v_{h}))^{2} -\int_{T} (\nabla \cdot \boldsymbol{\beta})\Pi_{k}(v_{h}) \, v_{h} \, dT  \nonumber\\
                                                      &+& \int_{T}(\nabla \cdot \boldsymbol{\beta}) \, \Pi_{k}(v_{h}) v_{h} dT - \int_{T} (\nabla.\boldsymbol{\beta}) \, v_{h}^{2} \, dT \Big| \nonumber\\
                                                      & = & \frac{1}{2} \Big| \int_{T} (\nabla \cdot \boldsymbol{\beta}) \, \Pi_{k}(v_{h}) (\Pi_{k}(v_{h})-v_{h})  \, dT \nonumber\\
                                                      &+ & \int_{T} (\nabla \cdot \boldsymbol{\beta}) \cdot v_{h}(\Pi_{k}(v_{h})-v_{h}) \, dT \Big |  \nonumber\\
                                                      & =& \frac{1}{2} \left| \int_{T} (\nabla \cdot \boldsymbol{\beta}) \, v_{h} (\Pi_{k}(v_{h})-v_{h}) \right|  \nonumber\\
                                                      & \leq & C \parallel \nabla \cdot \boldsymbol{\beta} \parallel_{\infty,T} \parallel v_{h} \parallel_{0,T} \parallel v_{h}-\Pi_{k}(v_{h}) \parallel_{0,T} \nonumber \\
                                                      & \leq & C \parallel \nabla \cdot \boldsymbol{\beta} \parallel_{\infty,T} \parallel v_{h} \parallel_{0,T} \, h_{T} \, \vert v_{h} \vert_{1,T}
                                                      \label{tempc7}
\end{eqnarray}

Since $\|v_{h}\|_{0,T} \leq \| v_{h} \|_{1,T}$ and $\vert v_{h} \vert_{1,T} \leq \| v_{h} \|_{1,T} $ inequalities 
hold \cite{ciarlet1991basic} we can estimate

\begin{equation}
-\vert b^{T}_{h}(v_{h},v_{h})-b^{T}(v_{h},v_{h}) \vert \geq - \, C \parallel \nabla.\boldsymbol{\beta} \parallel_{\infty,T}  h_{T} \parallel v_{h} \parallel_{1,T}
\label{tempc8}
\end{equation}

Therefore 
\begin{eqnarray}
A_{h}(v_{h},v_{h}) &= &\sum_{T}A^{T}_{h}(v_{h},v_{h}) \nonumber \\
                   & \geq & \sum_{T} \alpha_{\ast} \, \eta \, \vert v_{h} \vert_{1,T}^{2}+\sum_{T}\text{min}(1,\gamma_{\ast}) \, c_{0} \, \| v_{h} \|_{0,T}^{2}  \nonumber\\
                   & - & \sum_{T} C \, \| \nabla \cdot \boldsymbol{\beta} \|_{\infty,T}  \, h_{T} \, \| v_{h} \|_{1,T}  \nonumber\\
                   &\geq & \sum_{T} \alpha_{T} \| v_{h} \|_{1}^{2} \nonumber \\
                   & \geq & \alpha\sum_{T} \| v_{h} \|_{1}^{2} \nonumber\\
                   & = & \alpha \| v_{h} \|_{1,h} \nonumber                 
\end{eqnarray}
where $\alpha=\text{min}(\alpha_{T})$, and
\begin{equation}
\alpha_{T}=\text{min} \{ \alpha_{\ast}\eta-C \parallel \nabla.\boldsymbol{\beta} \parallel_{\infty} h_{T} ,\text{min}(1,\gamma_{\ast})c_{0}-C \parallel \nabla.\boldsymbol{\beta} \parallel_{\infty}h_{T} \} \nonumber
\end{equation}
\subsection{Continuity of discrete bilinear form}
\begin{lem}
Under the assumption of the polynomial consistency and stability along with 
the coefficients $\mathbf{K}, \boldsymbol{\beta}, c $ the 
bilinear form $ A_{h} $ defined in equation (\ref{temp16}) is continuous. 
\label{41}
\end{lem}

\begin{proof}
Diffusive part $a_{h}^{T}(u_{h},v_{h})$ and reactive part $c^{T}_{h}(u_{h},v_{h}) $ of the 
bilinear form $A^{T}(u_{h},v_{h})$ are symmetric and hence they can be view as inner product 
in VE space $V^{T}_{h}$ over each element $T$. Convective part $b^{T}_{h}(u_{h},v_{h})$ is 
not symmetric and hence we cannot bound it like diffusive part and reactive 
part, but using properties of the projection operator we can simply bound it. Hence 
we conclude
\begin{eqnarray}
a^{T}_{h}(u_{h},v_{h}) & \leq  &(a^{T}_{h}(u_{h},u_{h}))^{\frac{1}{2}} (a^{T}_{h}(v_{h},v_{h}))^{\frac{1}{2}} \nonumber\\
                       & \leq  &\alpha^{\ast} (a^{T}(u_{h},u_{h}))^{\frac{1}{2}} (a^{T}(v_{h},v_{h}))^{\frac{1}{2}} \nonumber \\
                       & \leq & \alpha^{\ast} \| \mathbf{K} \|_{\infty} \| \nabla u_{h} \|_{0,T} \| \nabla v_{h} \|_{0,T}
                       \label{tempb10}
\end{eqnarray} 

similarly 
\begin{equation}
c^{T}_{h}(u_{h},v_{h}) \leq \gamma^{\ast} \| c \|_{\infty} \| u_{h} \|_{0,T} \| v_{h} \|_{0,T}
\end{equation}

\begin{eqnarray}
b^{T}_{h}(u_{h},v_{h}) & = &  \frac{1}{2} \Big( \int_{T}\boldsymbol{\beta} \, \Pi_{k-1}(\nabla u_{h}) \, \Pi_{k}(v_{h}) \, dT \nonumber \\
                        &- &\int \boldsymbol{\beta} \cdot \Pi_{k-1}(\nabla v_{h}) \, \Pi_{k}(u_{h})  \, dT \nonumber\\            
                        & - &\int_{T}(\nabla \cdot \boldsymbol{\beta}) \, \Pi_{k}(u_{h}) \, \Pi_{k}(v_{h}) \, dT \Big) \nonumber 
\end{eqnarray}
Again
\begin{equation}
 \int_{T}\boldsymbol{\beta} \, \Pi_{k-1}(\nabla u_{h}) \, \Pi_{k}(v_{h}) \, dT \leq C \, \|\boldsymbol{\beta}\| _{\infty} \|\nabla u_{h} \|_{0,T} \|v_{h} \|_{0,T}
\end{equation}

\begin{equation}
\int_{T} \boldsymbol{\beta} \cdot \Pi_{k-1}(\nabla v_{h}) \, \Pi_{k}(u_{h})  \, dT \leq C \, \|\boldsymbol{\beta} \| _{\infty} \| \nabla v_{h} \|_{0,T} \| u_{h} \|_{0,T}
\end{equation}

\begin{equation}
\int_{T}(\nabla \cdot \boldsymbol{\beta}) \, \Pi_{k}(u_{h}) \, \Pi_{k}(v_{h}) \, dT \leq C \, \| \boldsymbol{\beta} \|_{1,\infty} \| u_{h} \|_{0,T} \| v_{h} \|_{0,T}
\end{equation}

Thus, 
\begin{eqnarray}
A^{T}_{h}(u_{h},v_{h}) & = & a^{T}_{h}(u_{h},v_{h}) +b^{T}_{h}(u_{h},v_{h})+c^{T}_{h}(u_{h},v_{h}) \nonumber\\
                       & \leq & \alpha^{\ast} \| \mathbf{K} \|_{\infty} \| \nabla u_{h} \|_{0,T} \| \nabla v_{h} \|_{0,T} \nonumber \\
                       & + & C \| \boldsymbol{\beta} \|_{1,\infty} \| u_{h} \|_{1,T} \| v_{h} \|_{0,T} \nonumber \\ 
                       & + &\gamma^{\ast} \| c \|_{\infty} \| u_{h} \|_{0,T} \| v_{h} \|_{0,T}  \nonumber\\
                       & \leq & C_{T} \|v_{h} \|_{1,T} \| u_{h} \|_{1,T}
\end{eqnarray}

\begin{eqnarray}
A_{h}(u_{h},v_{h}) & = & \sum_{T} A^{T}_{h}(u_{h},v_{h}) \nonumber \\
                   & \leq  & \sum_{T} C_{T} \|v_{h} \|_{1,T} \| u_{h} \|_{1,T} \nonumber \\
                   & \leq & C\left(\sum_{T} \| u_{T} \|^{2}_{1,T}\right)^{\frac{1}{2}} \left(\sum_{T} \| v_{h} \|^{2}_{1,T}\right)^{\frac{1}{2}} \nonumber \\
                   & = & C \, \| u_{h} \|_{1,h} \| v_{h} \|_{1,h}
                   \label{tempb11}
\end{eqnarray}
 Therefore the bilinear form is continuous. 
\end{proof}
The bilinear form $A^{T}_{h}$  is  $\textit{discrete coercive}$ and 
$\textit{bounded or continuous} $ in discrete norm $\| \cdot \|_{1,h}$ on 
nonconforming virtual element space $ V^{k}_{h}$, defined in equation (\ref{temp14}). Hence the bilinear 
form has unique solution in $V^{k}_{h,g}$ by Banach Necas Babuska(BNB) 
theorem \cite{di2011mathematical}.

\section{Convergence analysis and apriori error analysis in $\parallel.\parallel_{1,h}$ norm }
\label{step5}
In this section we reveal nonconforming convergence analysis of discrete solution 
$u_{h} \in V^{k}_{h}$ which satisfy the discrete bilinear form equation (\ref{temp10}). The basic 
idea is same as non conforming error analysis of convection diffusion problem proposed by 
Tobiska et al \cite{john1997nonconforming,knobloch2003p}. It is well known that 
non conforming virtual element space 
$V^{k}_{h} \subset H^{1,nc}(\tau_{h};k)\nsubseteq H^{1}(\Omega)$ and introduce 
consistency error. The finite element solution $ v_{h} \in V^{k}_{h}$ is not continuous 
along interior edge $e$ except certain points which implies an additional jump 
term $ J(u,v_{h})$ defined in equation (\ref{temp28}). Nonconforming virtual 
element space $V^{k}_{h}$ satisfy 'patch-test'\cite{irons1972experience} of order $k$. Using 
this property we can easily bound the jump term.

\begin{thm}
Let $u$ be the exact solution problem (\ref{temp3}) with polynomial 
coefficients $\mathbf{K}, \boldsymbol{\beta}, c$ . Let $u_{h} \in V^{k}_{h}$ be the solution 
of the non conforming virtual element approximation(\ref{temp10}). Let $f \in L^2({\Omega})$ 
and $ u \in H^{k+1}(\Omega)$ \,($k \geq 1$) then 
\begin{equation}
\| u-u_{h} \|_{1,h} \leq C h^{k} \| u \|_{k+1,h} +C \, h \, \vert f \vert_{0,h}
\end{equation}
where $\| \cdot \|_{1,h} $ denote broken norm in the space $\mathbf{H}^{1,NC}(\tau_{h},k)$. 
\end{thm}  

\begin{proof}
We consider $u_{I}$ be the approximation of $u$ in $V^{k}_{h}$ and $u_{\Pi}$ be 
polynomial approximation of $u$ in $\mathbb{P}^{k}(\tau_{h})$.

Define $ \delta:=u_{h}-u_{I}$.
using coercivity of the virtual element  form, we can write

\begin{eqnarray}
\alpha \| \delta \|^{2}_{1} & \leq & A_{h}(\delta,\delta) \nonumber \\
                                          & = & A_{h}(u_{h},\delta)-A_{h}(u_{I},\delta) \nonumber \\
                                          & = & <f_{h},\delta>-A_{h}(u_{I},\delta) \nonumber \\
                                          & = & <f_{h},\delta>-\sum_{T}A^{T}_{h}(u_{I},\delta) 
                                          \label{temp42}
\end{eqnarray} 
Now we shall analyse the local bilinear form $A^{T}_{h}(u_{I},\delta) $ term by term

\begin{eqnarray}
A^{T}_{h}(u_{I},\delta) & = & A^{T}_{h}(u_{I}-u_{\Pi},\delta)+A^{T}_{h}(u_{\Pi,\delta}) \nonumber \\
                        & = & A^{T}_{h}(u_{I}-u_{\Pi},\delta)+A^{T}_{h}(u_{\Pi},\delta)-A^{T}(u_{\Pi},\delta)+A^{T}(u_{\Pi},\delta)  \nonumber \\
                        & = & A^{T}_{h}(u_{I}-u_{\Pi},\delta)+A^{T}_{h}(u_{\Pi},\delta)-A^{T}(u_{\Pi},\delta) \nonumber \\
                        & + & A^{T}(u_{\Pi}-u,\delta)+A^{T}(u,\delta) 
                        \label{tempb12}
\end{eqnarray}

Discrete bilinear form $A^{T}_{h}(u_{h},v_{h}) $ is polynomial consistent, hence 
$A^{T}_{h}(u_{\Pi},\delta)=A^{T}(u_{\Pi},\delta) $

Therefore
\begin{equation}
A^{T}_{h}(u_{I},\delta)=A^{T}_{h}(u_{I}-u_{\Pi},\delta)+A^{T}(u_{\Pi}-u,\delta)+A^{T}(u,\delta)
\label{tempd3}
\end{equation}

 
Now using Green's theorem on the triangle $T$, we get 
\begin{eqnarray}
 \int_{T} (\boldsymbol{\beta} \cdot \nabla u) \, \delta \, dT &= &\frac{1}{2} \Big(\int_{T}(\boldsymbol{\beta} \cdot \nabla u) \, \delta \, dT -\int_{T} (\boldsymbol{\beta} \cdot \nabla \delta )u  \, dT \nonumber\\
                                                 &-&\int_{T}(\nabla \cdot \boldsymbol{\beta})u \, \delta  \, dT+\int_{\partial T} (\boldsymbol{\beta} \cdot \mathbf{n}_{\partial T})u \, \delta \, ds  \Big)                                                   
                                                 \label{tempd1}
\end{eqnarray}
Rearranging terms we can write
\begin{eqnarray}
\frac{1}{2} \left(\int_{T}(\boldsymbol{\beta} \cdot \nabla u) \, \delta \, dT -\int_{T} (\boldsymbol{\beta} \cdot \nabla \delta )u -\int_{T}(\nabla \cdot \boldsymbol{\beta}) \, u \, \delta  \, dT \right) & =& \int_{T} (\boldsymbol{\beta} \cdot \nabla u)\delta \, dT \nonumber\\
                                                  &-&\frac{1}{2}\int_{\partial T} (\boldsymbol{\beta} \cdot \mathbf{n}_{\partial T})u \, \delta \, ds 
                                                  \label{tempd2}
\end{eqnarray}
 
Taking sum over all element $ T \in \tau_{h}$ we get
 
\begin{eqnarray}
 \sum_{T}A^{T}(u,\delta) & = & \sum_{T} (-\nabla(\mathbf{K} \, \nabla u)+(\boldsymbol{\beta} \cdot \nabla u)+cu) \, \delta)  \nonumber\\
                         & + &\sum_{T} \int_{\partial T} \mathbf{K} \, \nabla u \cdot \mathbf{n}_{e} \, \delta \, ds -\sum_{T}  \frac{1}{2}\int_{\partial T} (\boldsymbol{\beta} \cdot \mathbf{n}_{\partial T})u \, \delta \, ds \nonumber\\
                         & = & <f,\delta> +\sum_{e \in \varepsilon_{h}} \int_{e} (\mathbf{K} \, \nabla u \cdot \mathbf{n}_{e}) \, [\vert \delta \vert] \, ds-\frac{1}{2} \sum_{e \in \varepsilon_{h}} \int_{e} (\boldsymbol{\beta}.\mathbf{n}_{e}) \, u \, [\vert \delta \vert] 
                         \label{temp44}
\end{eqnarray}
 
In the above equation we get jump term corresponding to $\delta $ only, since $\delta$ is 
discontinuous along interior edge $ e \in \varepsilon^{0}_{h}$  
  
Globally the equation (\ref{tempd3}) can be written as
  
\begin{eqnarray}
 A_{h}(u_{I},\delta) & = & \sum_{T} A^{T}_{h}(u_{I},\delta) \nonumber \\
                     & = & \sum_{T} A^{T}_{h}(u_{I}-u_{\Pi},\delta) +\sum_{T}A^{T}(u_{\Pi}-u,\delta) \nonumber \\
                     & + & <f,\delta> +\sum_{e \in \varepsilon_{h}} \int_{e}(\mathbf{K} \, \nabla u \cdot \mathbf{n}_{e}) \, [\vert \delta \vert] \, ds -\frac{1}{2}  \sum_{e \in \varepsilon_{h}} \int_{e} (\boldsymbol{\beta} \cdot \mathbf{n}_{e}) \, u \, [\vert \delta \vert] 
                     \label{temp45}
\end{eqnarray}
Let us denote 
\begin{eqnarray}
 M_{1} & =  & A^{T}_{h}(u_{I}-u_{\Pi},\delta) \nonumber \\
 M_{2} & = & A^{T}(u_{\Pi}-u,\delta) \nonumber
\end{eqnarray}

\begin{equation}
 \vert A^{T}_{h}(u_{I}-u_{\Pi},\delta) \vert \leq C \| u_{I}-u_{\Pi} \|_{1,T} \| \delta \|_{1,T} 
 \label{temp46}
\end{equation}

\begin{equation}
\vert A^{T}(u_{\Pi}-u,\delta) \vert \leq C \| u_{\Pi}-u \|_{1,T} \| \delta \|_{1,T}
\label{temp47}
\end{equation}

using (\ref{temp46}), (\ref{temp47}) and interpolation error estimation (\ref{temp15}) we 
can bound 
 
\begin{eqnarray}
 \vert M_{1} \vert +\vert M_{2} \vert & \leq  & C \, \|\delta \, \|_{1,T} (\| u_{I}-u_{\Pi} \|_{1,T}+ \|u_{\Pi}-u \|_{1,T} )  \nonumber\\
                                       & \leq & C \, \| \delta \, \|_{1,T} \, h^{k} \, \| u \|_{k+1,T}
                                       \label{temp48}  
\end{eqnarray}
 
we use equation (\ref{temp25}) for $k\geq2$ and (\ref{temp27}) for $k=1$ to bound the 
right hand side
\begin{eqnarray}
 \vert <f_{h},\delta>-<f,\delta> \vert  & = & \left| \sum_{T} \int_{T} (f_{h}-f) \, \delta \right| \nonumber\\
                                        & \leq & C h^{k} \vert f \vert_{k-1,\tau_{h}} \vert \delta \vert_{1,h} 
                                        \label{temp49}
\end{eqnarray}
 
In particular for $k=1$
\begin{equation*}
 \left| <f_{h},\delta> - <f,\delta> \right| \leq C \, h \, |f|_{0,\tau_{h}} \, \delta_{1,h}
\end{equation*}

 using  estimation (\ref{tempb6}) we bound consistency error 
\begin{eqnarray}
 \left| \int_{e} (\mathbf{K} \, \nabla u \cdot \mathbf{n}_{e}) [\vert\delta\vert]  \right| & \leq & \| \mathbf{K} \|_{\infty} \vert \int_{e} (\nabla u \cdot \mathbf{n}_{e}) [\vert\delta\vert] \vert \nonumber \\
                                                                   & \leq & C h^{k} \| u \|_{k+1,T^{+}\cup T^{-}} \vert \delta \vert_{1,T^{+}\cup T^{-}} \nonumber 
\end{eqnarray}
Edge $e$ is an interior edge shared by triangles $T^{+}$ and $T^{-}$. Therefore
\begin{equation}
 \left| \sum_{e \in \varepsilon_{h}} \int_{e}(\mathbf{K} \, \nabla u \cdot \mathbf{n}_{e}) [\vert\delta \vert] \right| \leq C \, h^{k} \, \| u \|_{k+1,h} \|\delta \|_{1,h}
 \label{tempm1}
\end{equation}
 again 
\begin{eqnarray}
 \left| \sum_{e\in \varepsilon_{h}} \int_{e} (\boldsymbol{\beta} \cdot \mathbf{n}_{e}) \, u \, [\vert \delta \vert]\right| & \leq &
  \left| \sum_{e \in \varepsilon_{h}} \| \boldsymbol{\beta} \cdot \mathbf{n}_{e} \|_ {\infty} \int_{e} u \, [\vert \delta \vert] \ \right| \nonumber\\
  & \leq & C \, \left| \sum_{e \in \varepsilon_{h}} \int_{e}(u-P^{k-1}(u)) \, [\vert \delta \vert] \ \right| \nonumber\\
  & = & C \, \left| \sum_{e \in \varepsilon_{h}} \int_{e}(u-P^{k-1}(u)) \, ([\vert \delta \vert]-P^{0}([\vert \delta \vert])) \ \right| \nonumber\\
  & \leq & C \sum_{e \in \varepsilon_{h}} \| u-P^{k-1}(u) \|_{0,e} \|[\vert \delta \vert]-P^{0}([\vert \delta \vert]) \|_{0,e} \nonumber
\end{eqnarray}
 
using standard approximation \cite{ciarlet1991basic}
\begin{eqnarray}
 \| u-P^{k-1}(u) \|_{0,e} & \leq & C h^{\text{min}(k,s)-1/2} \| u \|_{s,T^{+}\cup T^{-}} \nonumber \\
 \|[\vert \delta \vert]-P^{0}([\vert \delta \vert]) \|_{0,e} & \leq & C h^{1/2} \| \delta \|_{1,T^{+}\cup T^{-}} \nonumber
\end{eqnarray}
 
we can bound 
 
\begin{eqnarray}
\left| \sum_{e\in \varepsilon_{h}} \int_{e} (\boldsymbol{\beta} \cdot \mathbf{n}_{e})u [\vert \delta \vert] \right| &\leq & C h^{\text{min}(k,s)}\sum_{T} \| u \|_{s,T} \| \delta \|_{1,T} \nonumber\\
 & \leq & C h^{\text{min}(k,s)} \| u \|_{s+1,h} \| \delta \|_{1,h}  \nonumber
\end{eqnarray}
 
In particular for $s=k$
\begin{equation}
\left| \sum_{e\in \varepsilon_{h}} \int_{e} (\boldsymbol{\beta} \cdot \mathbf{n}_{e}) \, u \, [\vert \delta \vert]\right| \leq C \, h^{k} \, \| u \|_{k+1,h} \| \delta \|_{1,h} 
\label{tempn1}
\end{equation}
 
Using (\ref{temp48}), (\ref{temp49}), (\ref{tempm1}), (\ref{tempn1}) we bound
\begin{eqnarray}
 \alpha \| \delta \|_{1,h}^{2} & \leq &  (C h^{k} \|u \|_{k+1,h}+C \, h \, \vert f \vert_{0,h}) \| \delta \|_{1,h}\nonumber\\
 \| \delta \|_{1,h} & \leq & C h^{k} \| u \|_{k+1,h} +C h \vert f \vert_{0,h} \label{tempn2}
\end{eqnarray}
 
We can write 
\begin{equation} 
 \|{u-u_{h}} \|_{1,h}\leq \|(u-u^{I}) \|_{1,h}+ \|(u^{I}-u_{h}) \|_{1,h}
\end{equation}

The first term can be estimated using the standard approximation (\ref{temp15}) 
and second term can be estimated using (\ref{tempn2}). Hence we obtain
\begin{equation}
\| u-u_{h} \|_{1,h} \leq C h^{k} \| u \|_{k+1,h}+C \, h \, \vert f \vert_{0,h}
\end{equation}
\end{proof}

\section{Conclusions}
\label{last1}
In this work we presented the analysis of nonconforming virtual element
method for convection diffusion reaction equation with polynomial
coefficients using $L^2$ projection. $L^2$ projection can be partially computed using 
degrees of freedom of the finite element. The external virtual element method using 
$L^2$ projection operator is not fully computable which may be considered
as a drawback of this method. We have proved stability of the method
assuming that the model problem is diffusion dominated. If the model
problem is convection dominated then the present analysis is not applicable
and hence we require a new framework for convection dominated problem, which
may be carried out as a future work.




\bibliographystyle{plain}
\bibliography{nonvirtl_elmt}

\end{document}